\documentclass[12pt]{amsart} 
\usepackage{amsthm,amsbsy,amsfonts,amssymb,amscd}

\usepackage[mathscr]{euscript} 

\usepackage[T1]{fontenc} 
\usepackage{sidenotes} 

\usepackage{enumerate} 
\usepackage{tikz} 
\usetikzlibrary{calc,intersections,through, backgrounds,arrows,decorations.pathmorphing,backgrounds,positioning,fit,petri} 
\usetikzlibrary{calc}

\usepackage{tikz-cd} 

\usepackage{hyperref}

\usepackage[margin=1in]{geometry}

\newcommand{\dbC}{{\mathbb{C}}} 

\newcommand{\dbG}{{\mathbb{G}}}

\newcommand{\dbZ}{{\mathbb{Z}}}

\newcommand{\scrE}{{\mathscr{E}}}

\newcommand{\scrH}{{\mathscr{H}}}

\newcommand{\scrO}{{\mathscr{O}}} 
\newcommand{\scrP}{{\mathscr{P}}} 
\newcommand{\scrR}{{\mathscr{R}}}

\newcommand{\scrV}{{\mathscr{V}}} 
\newcommand{\scrW}{{\mathscr{W}}}

\newcommand{\bfe}{\mathbf{e}}

\newcommand{\bfV}{\mathbf{V}}

\newcommand{\id}{\mathrm{id}}

\newcommand{\GL}{\mathrm{GL}}

\newcommand{\Hom}{\mathrm{Hom}}

\DeclareMathOperator{\spec}{Spec} 

\DeclareMathOperator{\vol}{vol}

\newcommand{\aut}{\mathrm{Aut}} 

\DeclareMathOperator{\sym}{Sym}

\newcommand{\dfs}{{/\kern-2pt/}}

\numberwithin{equation}{subsection}

\newtheoremstyle{notes} {} {} {} {} {\bfseries} {.} {.5em} {}

\theoremstyle{plain}

\newtheorem{prop}[subsection]{Proposition}

\newtheorem{lemma}[subsection]{Lemma}

\newtheorem{cor}[subsection]{Corollary}

\newtheorem{thm}[subsection]{Theorem}

\theoremstyle{remark}

\newtheorem{rems}[subsection]{Remarks} 

\newtheorem{rem}[subsection]{Remark} 

\swapnumbers

\theoremstyle{remark} 
 
\newtheorem*{nota}{Notation} 
\newtheorem{ex}[subsection]{Example}

\newtheorem{pf34}[subsection]{Proof of~\ref{3.4}}

\newtheorem{pf54}[subsection]{Proof of~\ref{5.4}}

\newtheoremstyle{construction} {} {} {} {} {\bfseries} { } {0pt} {}

\theoremstyle{construction}

\title[Orthosymplectic super groups]{The first fundamental theorem of invariant theory for the orthosymplectic super group} 

\author{P.~Deligne} 

\author{G.I.~Lehrer}

\author{R.B.~Zhang}

\newcommand{\ov}{\mathrm{O}(V)}

\newcommand{\osv}{\mathrm{O}(V, q)}

\newcommand{\osp}{\mathrm{OSp}}

\newcommand{\scrhom}{\scrH\mathrm{om}}

\newcommand{\vti}{V \times V \times \dotsb \times V }

\newcommand{\bfhom}{\mathbf{Hom}}

\newcommand{\bfU}{\mathbf{U}}

\newcommand{\bfW}{\mathbf{W}}

\newcommand{\bff}{\mathbf{f}}

\newcommand{\bsV}{\boldsymbol{\scrV}}



\begin{document} 
\date{} 
\begin{abstract}
	We give a new proof, inspired by an argument of Atiyah, Bott and Patodi, of the first fundamental theorem of invariant theory for the orthosymplectic super group. We treat in a similar way the case of the periplectic super group. Lastly, the same method is used to explain the fact that Sergeev's super Pfaffian, an invariant for the special orthosymplectic super group, is polynomial. 
\end{abstract}

\maketitle 

\tableofcontents

\section{Introduction} 
\subsection{}\label{1.1}

Let $V$ be a finite dimensional complex vector space, with a non degenerate symmetric bilinear form $B$. The first fundamental theorem (FFT) for the orthogonal group $\ov$ gives generators for the linear space of $\ov$-invariant multilinear forms on $\vti$ ($N$ factors). The generators are obtained from partitions $\scrP$ of $\{1,\dotsc, N\}$ into subsets with two elements: to each $\scrP$ corresponds the multilinear form 
\begin{equation}
	\label{1.1.1} v_1,\dotsc, v_N \mapsto \prod_{\{i_1, i_2\} \in \scrP} B(v_{i_1},v_{i_2}). 
\end{equation}

This theorem is proved in Weyl \cite{Weyl} using the Capelli identity. In appendix 1 to \cite{abp}, Atiyah, Bott and Patodi suggested a more geometric approach, reducing the FFT for $\ov$ to the FFT for $\GL(V)$. The latter gives generators for the linear space of $\GL(V)$-invariant multilinear forms on $V \times \dotsb \times V \times V^\vee \times \dotsb \times V^\vee$: the generators arise from pairings between the factors $V$ and $V^\vee$, by a formula similar to~\eqref{1.1.1}. The FFT for $\GL(V)$ is equivalent to the Schur-Weyl duality between $\GL(V)$ and the symmetric group~$S_M$, both acting on~$V^{\otimes M}$.

\subsection{} We show that this idea can be made to work in the ``super world'', where one systematically considers $\dbZ/2$-graded objects, and where for super vector spaces (that is, $\dbZ/2$-graded vector spaces) $V$ and $W$, the isomorphism $V \otimes W \to W \otimes V$ is defined to be 
\begin{equation}
	\label{1.2.1} v \otimes w \mapsto (-1)^{|v||w|}w \otimes v 
\end{equation}
for $v$ and $w$ homogeneous of degrees $|v|$ and $|w|$ (sign rule). For an explanation of how concepts of algebraic or differential geometry extend to the super world, we refer to Leites \cite{dl}, Manin \cite{ym}, or Bernstein, Deligne and Morgan \cite{pm}, as well as to \S2 below.

In the super world, $V$ is taken to be a super vector space, and $B$ to be a non degenerate bilinear form, symmetric in the sense of being invariant by~\eqref{1.2.1}. The FFT describes the multilinear forms on $\vti$, invariant by the algebraic super group $\ov$. They are derived from~$B$ by~\eqref{1.1.1}. This was announced by Sergeev in~\cite{as}, and we obtain a new proof of this FFT in \S3. Another approach, also inspired by~\cite{abp}, using as ring of coefficients the infinite dimensional Grassmann algebra, may be found in~\cite{lz14a}. In~\cite{lz14b}, this reduction to the case of $\GL(V)$ leads to a new second fundamental theorem (SFT) which describes all relations among the generators.

The form $B$ can be viewed as a morphism
\[ V \otimes V \to \dbC \]
invariant by~\eqref{1.2.1}. ``Morphism'' implies ``compatible with the $\dbZ/2$-grading''. As
\[ (V \otimes V)_0 = V_0 \otimes V_0 \oplus V_1 \otimes V_1, \]
and $\dbC$ is purely of degree~$0$, $B$ is given by non degenerate bilinear forms on the vector spaces $V_0$~and~$V_1$. The symmetry condition is that $B$ be symmetric on~$V_0$ and antisymmetric on~$V_1$. Because of this, $\ov$ is called orthosymplectic and also denoted $\mathrm{OSp}(V)$.

\subsection{} In \S4, we apply the same idea to a super vector space~$V$ equipped with an odd symmetric non degenerate bilinear form~$B$. Let $\Pi$ be $\dbC$, with the $\dbZ/2$-grading for which it is purely odd. The tensor product with~$\Pi$ is the parity change functor. The form $B$ is a morphism $V \otimes V \to \Pi$, invariant by~\eqref{1.2.1}. The FFT for the algebraic super group $\pi\ov := \aut(V, B)$ says that $\pi\ov$-invariant multilinear forms on $V \times \dotsb \times V$, with values in $\Pi$~or~$\dbC$, are generated by forms~\eqref{1.1.1}.

\subsection{} Let us return to the orthosymplectic case, with $\dim V = m | 2n$, that is $\dim V_0 = m$ and $\dim V_1 = 2n$. If $m > 0$, $\mathrm{OSp}(V)$ has two connected components, because $\mathrm{O}(V_0)$ has two connected components while $\mathrm{Sp}(V_1)$ is connected. Let $\mathrm{SOSp}(V)$ be the connected component of the identity. In the classical case of a purely even $V$, $\bigwedge^m V = \bigwedge^m V_0$ is the trivial representation of~$\mathrm{SO}(V)$, and for any isomorphism $\vol: \bigwedge^m V \to \dbC$,
\[ v_1,\dotsc, v_m \mapsto \vol (v_1 \wedge \dotsb \wedge v_m) \]
is an $\mathrm{SO}(V)$-invariant multilinear form which is not $\ov$-invariant. Up to a factor depending on $\vol$, its square is
\[ \det(B(v_i, v_j)).\]
In the super world, a similar $\mathrm{SOSp}(V)$-invariant is a square root of $\det(B(v_i, v_j))^{2n+1}$ which Sergeev calls the super pfaffian. We explain the existence of such a polynomial invariant in~\S5. See \cite{lz15} for an alternative approach to the proof of its basic properties, as well as a discussion of the invariants of~$\mathrm{SOSp}(V)$.

\section{Tools from algebraic geometry}

\subsection{}\label{2.1} Let us explore where the sign rule \eqref{1.2.1} leads us. A \emph{commutative super ring} is a mod $2$-graded ring (with unit) such that 
\begin{equation}
	\label{2.1.2} ab = (-1)^{|a||b|} ba 
\end{equation}
and that $a^2 = 0$ for $a$ odd (a consequence of \eqref{2.1.2} if $2$ is invertible). Starting with commutative super rings, the definitions of \emph{super schemes, super schemes over a field~$k$, super group schemes over~$k$, \ldots} are modeled on the parallel notions in algebraic geometry, and many theorems admit a straightforward generalization. We will use freely theorems in super algebraic geometry whose proof is an easy extension of the proof of a parallel classical theorem. This applies to the formalism of flat, etale, smooth, unramified maps, and to the formalism of faithfully flat descent as expounded in SGA1 (talks I--III and VIII). Some definitions are recalled below as well as in \ref{2.8}.

A commutative super ring $R$ has a \emph{spectrum} $\spec(R)$. This spectrum is a topological space endowed with a sheaf of commutative super rings, the \emph{structural sheaf~$\scrO$.} \emph{Super schemes} are obtained by gluing such spectra. A \emph{closed subscheme} of~$\spec(R)$ is a $\spec(R/I)$, for~$I$ a mod $2$-graded ideal of~$R$.

An equivalent definition of super schemes is that they are topological spaces~$S$ endowed with a sheaf of commutative super rings, the \emph{structural sheaf~$\scrO$,} or $\scrO_S$, such that $(S, \scrO_0)$ is a scheme and $\scrO_1$ a quasi-coherent $\scrO_0$-module. An \emph{open subscheme} is an open subset endowed with the restriction to it of the structural sheaf. A super scheme~$S$ is \emph{over~$k$} if it comes with a morphism $S \to \spec(k)$, that is if $\scrO$ is a $k$-algebra.

The ordinary schemes form a full subcategory of the category of super schemes: they are the super schemes for which $\scrO_1 = 0$. Any super scheme~$S$ contains a largest (closed) ordinary scheme~$S_{\mathrm{rd}}$, with structural sheaf the quotient of $\scrO_S$ by the sheaf of ideals $(\scrO_1)$ generated by~$\scrO_1$. For $T$ an ordinary scheme, $\Hom(T,S_{\mathrm{rd}}) \xrightarrow{\sim} \Hom (T,S)$. The ideal~$(\scrO_1)$ being nilpotent, one can picture $S$ has being the ordinary scheme~$S_{\mathrm{rd}}$, surrounded by a nilpotent fuzz. Topological notions (the underlying topological space, its (Krull) dimension, Zariski density of subsets,\ldots) depend only on~$S_{\mathrm{rd}}$.

As in algebraic geometry, to capture the geometric meaning of these definitions, one should consider the corresponding representable functors. See below. 
\begin{nota}
	We fix a field~$k$. From \S3 on, $k = \dbC$. We will work in the category of super schemes~$S$ over~$k$.
	
	Super vector spaces will as a rule be assumed to be finite dimensional, and super schemes to be of finite type over~$k$. Group schemes of finite type over~$k$ will be called \emph{algebraic groups.}
	
	When this does not lead to ambiguities, we will drop the adjective ``super'', or, as a reminder, write it parenthetically. 
\end{nota}

\subsection{}\label{2.2} 

If $X$ and $S$ are over~$k$, an \emph{$S$-point} of~$X$ is a morphism from $S$~to~$X$. Intuitively: a family of points of~$X$ parametrized by~$S$. The \emph{functor of points $h_X$} of~$X$ is the contravariant functor
\[ h_X : S \mapsto X(S) := \text{the set $\Hom(S,X)$ of $S$-points of~$X$.}\]
By Yoneda's lemma, the functor $X \to h_X$ is fully faithful. It will often be convenient to define a morphism $X \to Y$ by defining its effect on $S$-points: by defining a morphism of functors $h_X \to h_Y$.

For $U$ running over the open subsets of~$S$ (with the induced (super) scheme structure), $U \mapsto X(U)$ is a sheaf on~$S$. Because of this, $h_X$ is determined by its restriction to affine schemes (those of the form $\spec(R)$): $X$ is determined by the covariant \emph{functor of $R$-points}
\[ R \mapsto X(R) := X(\spec(R))\]
and to define a morphism $X \to Y$, it suffices to define a functorial morphism $X(R) \to Y(R)$. 
\begin{ex}
	\label{2.3} Let $V$ be a (finite dimensional, super) vector space over~$k$. The functor of points of the corresponding \emph{affine space~$\bfV$} is
	\[ R \mapsto \text{even component of $V_R := R \otimes V$.}\]
	If $V = k^{p|q}$, that is if $V_0 = k^p$ and $V_1 = k^q$, this affine super space is the \emph{standard affine space} $A^{p|q}$.
	
	By \emph{``basis''}, we will always mean homogeneous basis: in a basis $(e_i)$, each $e_i$ is either even, or odd. If $(e_i)$ is a basis of~$V$, the $R$-points $x$ of $\bfV$ are the $\sum x^i e_i$ with $x^i$ in~$R$ of the same parity as $e_i$. The $x^i$ are the \emph{coordinates} of~$x$ in the basis~$(e_i)$ and $\bfV$ is the spectrum of $k[(X_i)]$, freely generated over~$k$ by indeterminates $X_i$, with $X_i$ of parity equal to that of~$e_i$. Suppose $V$ is in duality with~$W$. If $(e^i)$ is the basis of~$W$ dual to the basis~$(e_i)$ of~$V$, meaning that $\langle e_i, e^j \rangle = \delta_i^j$, the coordinates of~$x \in \bfV(R)$ are the $\langle x, e^i \rangle$, and $\bfV = \spec(\sym^*(W))$, for $\sym^*(W)$ the commutative super $k$-algebra freely generated by the super vector space~$W$. This algebra is $\dbZ$-graded, with~$W$ the component of degree one. 
\end{ex}
\begin{ex}
	\label{2.4} If $V = k$, purely even, we get the \emph{additive group scheme} $\dbG_a = \spec k[X]$ with functor of points $R \to R_0$. The group law is $+$.
	
	If $V = k$, purely odd, we get to the \emph{odd additive group} $\dbG^-_a = \spec(k[\theta])$, with $\theta$ odd and hence $\theta^2 = 0$. The functor of points is $R \mapsto R_1$. The group law is $+$.
	
	More generally, for any $X$ over~$k$, even (resp. odd) global sections of the structural sheaf~$\scrO$ can be identified with morphisms $X \to \dbG_a$ (resp. $X \to \dbG^-_a$). Passing to a corresponding functors of points, we get an interpretation of even (resp. odd, resp. arbitrary) global sections of $\scrO$ as functions~$f$ which to each $R$-point $x$ of~$X$ functorially in~$R$ attach a value $f(x)$ in~$R_0$ (resp. $R_1$, resp. $R$). The global sections of the structural sheaf $\scrO$ of~$X$ will be called \emph{functions on~$X$,} and the above explains how this terminology is to be interpreted. 
\end{ex}
\begin{ex}
	\label{2.5} Let $V$ and $W$ be (super) vector spaces. The \emph{inner Hom} $\scrhom(V,W)$ is the vector space of all linear maps from $V$ to $W$, with its natural $\dbZ/2$-grading. Let $\bfhom(V,W)$ be the corresponding affine space. The $R$-points of~$\bfhom(V, W)$ are the (super) $R$-module morphisms from $V_R$ to~$W_R$. 
\end{ex}
\begin{ex}
	\label{2.6} Let $U$, $V$ and $W$ be (super) vector spaces, and let 
	\begin{equation}
		\label{2.6.1} f : U \otimes V \to W 
	\end{equation}
	be a morphism. After extension of scalars to~$R$, $f$ induces $f_R : U_R \otimes_R V_R \to W_R$, and a map 
	\begin{equation}
		\label{2.6.2} \bfU(R) \times \bfV(R) \to \bfW(R). 
	\end{equation}
	Applying Yoneda's lemma, we get from \eqref{2.6.2} a morphism 
	\begin{equation}
		\label{2.6.3} \bff : \bfU \times \bfV \to \bfW. 
	\end{equation}
	
	The morphism~$\bff$ is bilinear, meaning that for any $R$, \eqref{2.6.2} is $R_0$-bilinear. We leave it as an exercise to the reader to check that conversely, any bilinear morphism~\eqref{2.6.3} is induced by a unique morphism~\eqref{2.6.1}. (cf.~\cite{pm} 1.7)
	
	The condition that $\bff$ be bilinear is implied by it being homogeneous of degree one in each variable. Writing~$\bfW$ as a product of~$\dbG_a$ and $\dbG^-_a$, this claim reduces to the case of functions on~$\bfU \times \bfV$. If $U$ and $U'$, $V$ and $V'$ are in duality, as in \ref{2.3}, a function $f$ on~$\bfU \times \bfV$, that is $f$ in $\sym^*(U') \otimes \sym^*(V')$, is homogeneous of degree~$n$ in the first variable, that is $f(r_0 x, y) = r_0^nf(x, y)$ for $R$-points $x$ and $y$, and $r_0 \in R_0$, if and only if it is in $\sym^n (U') \otimes \sym^*(V')$. Similarly for the second variable. Homogeneity of degree one of $\bff$ in both variables means that $\bff$ is in~$U' \otimes V'$, hence bilinear. 
\end{ex}
\subsection{}\label{2.7} With the notation of~\ref{2.3}, a \emph{quadratic form} $q$ on~$V$ is an even function~$q$ on $\bfV$ homogeneous of degree~$2$, meaning that for any $R$-point $x$ of~$\bfV$ and $r_0$ in $R_0$, $q(r_0 x) = r_0^2q(x)$. Equivalently, $q$ is in the even part of $\sym^2(W) \subset \Gamma(\bfV, \scrO)$. If $q$ is a quadratic form, the associated $B : \bfV \times \bfV \to \dbG_a$, defined by 
\begin{equation}
	\label{2.7.1} B(x, y) = q(x + y) - q(x) - q(y), 
\end{equation}
is symmetric bilinear and $q(x) = \frac{1}{2}B(x,x)$. By~\ref{2.6}, \eqref{2.7.1} is induced by a morphism, also noted $B$: 
\begin{equation}
	\label{2.7.2} B: V \otimes V \to k. 
\end{equation}

\subsection{}\label{2.8} \emph{Etale} and \emph{unramified} morphisms $f : X \to Y$ are defined as in SGA1. Their counterparts in differential geometry are local isomorphisms and immersions. A morphism $f : X \to Y$ is \emph{smooth} of relative dimension~$p|q$ if locally on~$X$ it factors through an etale map to $Y \times A^{p|q}$: 
\begin{equation}
	\label{2.8.1} X \supset U \longrightarrow Y \times A^{p|q} \longrightarrow Y. 
\end{equation}
The counterpart in differential geometry of ``smooth'' is ``submersion''. A factorization~\eqref{2.8.1} is a (relative) \emph{local coordinate system.}

\subsection{}\label{2.9} The constructions \ref{2.3} to~\ref{2.7} can be repeated over a scheme~$S$. A \emph{vector bundle} $\scrV$ over~$S$ is a sheaf of $\scrO$-modules locally free of finite type. More precisely, as we are in the super world, it is a sheaf of $\dbZ/2$-graded modules~$\scrV$ over the $\dbZ/2$-graded structural sheaf~$\scrO$, locally admitting a basis~$(e_i)$. As in~\ref{2.3}, it is always assumed that in a basis each $e_i$ is either even or odd. Locally on~$S$, $\scrV$ admits non canonical decompositions $\scrV = \scrV^+ \oplus \scrV^-$ such that $\scrV^+$ has a purely even basis, and $\scrV^-$ a purely odd one. Only when $S$ is an ordinary scheme, that is when $\scrO_1 = 0$, does such a decomposition agree with the grading $\scrV = \scrV_0 \oplus \scrV_1$.

If $s$ is an $R$-point of~$S$, an $R$-point over~$s$ of the corresponding \emph{affine space bundle} $\bsV$ is an even section of the pull-back of~$\scrV$ to~$\spec(R)$. If $\scrV$ and $\scrW$ are in duality, a quadratic form on~$\bsV$ is an even section of~$\sym^2(\scrW)$. The associated bilinear form~$B$ is bilinear, and corresponds to a morphism
\[ \scrV \otimes \scrV \to \scrO_S. \]
\begin{ex}
	\label{2.10} Let $V$ be a (super) vector space. The \emph{general linear group} $\GL(V)$ is the algebraic group with functor of points
	\[ R \mapsto \text{group of automorphisms of~$V_R$}. \]
	As a scheme, it is an open subscheme of $\bfhom(V,V)$ considered in~\ref{2.5}. 
\end{ex}
\begin{ex}
	\label{2.11} Assume we are not in characteristic~$2$. A quadratic form~$q$ on~$V$ is \emph{non degenerate} if the associated bilinear form~$B$ \eqref{2.7.2} is a self duality of~$V$. The corresponding \emph{orthosymplectic group} $\osv$ is the subgroup of~$\GL(V)$ respecting~$q$, or, what amounts to the same, respecting $B$. Define $\bfV_R$ to be the scheme $\bfV \times \spec(R)$ over $\spec(R)$. An $R$-point of $\osv$ is an automorphism~$g$ of~$V_R$ such that the induced automorphism of~$\bfV_R$ fixes the inverse image of~$q$ in~$\Gamma(\bfV_R, \scrO)$; equivalently, $g$ should fix the bilinear form $V_R \times V_R \to R$ induced by~$B$.
	
	The bilinear form~$B$ of~$V$ is symmetric, in the super sense. It is the orthogonal direct sum of a form $B^+$ on $V^+ := V_0$, such that
	\[ B^+(x, y) = B^+(y, x) \]
	and of a form $B^-$ on $V^- := V_1$, such that
	\[ B^-(x, y) := -B^-(y, x), \]
	hence $B^-(x,x) = 0$. This explains the terminology ``orthosymplectic''.
	
	If $q$ is non degenerate, after a separable extension $k'/k$ and in a suitable basis~$(e_i)$ (with $e_i$ even for $1 \leq i \leq m$, odd for $m + 1 \leq i \leq m + 2n$), $q$ can be written 
	\begin{equation}
		\label{2.11.1} q(x) = \frac{1}{2}\sum_1^m x_i^2 + \sum_1^n x_{m + i} x_{m + i + n}. 
	\end{equation}
	For the corresponding~$B$, one has 
	\begin{equation}
		\label{2.11.2} B(e_i, e_i) = 1 \quad (1 \leq i \leq m) 
	\end{equation}
	\[B(e_{m + i}, e_{m + i + n}) = -1 = -B(e_{m + i + n}, e_{m+i})\]
	and the other $B(e_i, e_j)$ vanish. To put $q$ in this canonical form, one may have to extract square roots, hence the need for the extension~$k'/k$. 
\end{ex}

\subsection{}\label{2.12} One can repeat \ref{2.10} and \ref{2.11} over a base~$S$. If $\scrV$ is a vector bundle over~$S$, $\GL(\scrV)$ is the group scheme over~$S$, whose $R$-points over an $R$-point $s$ of~$S$ is the group of automorphisms of the pull back of $\scrV$ by~$s$.

Assuming $2$ invertible on~$S$, the definition of non degenerate quadratic forms on $\scrV$ repeats \ref{2.11} 
\begin{prop}
	\label{2.13} Let $q$ be a non degenerate quadratic form on~$\scrV$. Locally on~$S$ for the etale topology, $\scrV$ admits a basis in which $q$ is given by~\eqref{2.11.1} 
\end{prop}

This means that $S$ can be covered by etale $f_i : U_i \to S$ such that, after pull back to each $U_i$, $q$ has the form~\eqref{2.11.1} in a suitable basis. 
\begin{proof}
	We may suppose that $\scrV$ is everywhere of some dimension~$m|2n$. The proof is by induction on $m + 2n$, the case $m = n = 0$ being trivial. If $m > 0$, the fiber~$V$ of~$\scrV$ at any $\dbC$-point $s$ contains $v$ even such that $q(v) \neq 0$. Let $e$ be an even local section passing through~$v$. In some open neighborhood~$U$ of~$s$, $q(e)$ is invertible. On an etale double cover of~$U$, $1/2q(e)$ has a square root~$r$. Replacing $e$ by $re$, we may assume that $q(e) = 1/2$. Because $q(e)$ is invertible, $\scrV$ is the direct sum of $\scrO e$ and of its orthogonal $\scrV'$ which is of dimension~$n-1|2m$. The local section~$e$ is the beginning of the desired basis. One completes it by applying the induction hypothesis to~$\scrV'$.
	
	If $m = 0$ and $n > 0$, one similarly finds odd local sections $e$~and~$f$ such that $B(e, f)$ is invertible. Dividing $f$ by $-B(e, f)$ one may assume that $B(e,f)=-1$. The symmetry of~$B$ ensures that $B(f,e) = 1$, and that $B(e,e) = B(f,f) = 0$. One concludes by applying the induction hypothesis to the orthogonal complement of~$\scrO e \oplus \scrO f$. 
\end{proof}

\subsection{}\label{2.14} The group scheme $\mathrm{O}(\scrV, q)$ over~$S$ is defined as $\osv$ is. Both $\GL(\scrV)$ and $\mathrm{O}(\scrV, q)$ are smooth over~$S$. Indeed, locally over~$S$ for the etale topology, they are pull backs of algebraic groups $\GL(V)$ and $\osv$ smooth over~$\spec(R)$. In characteristic zero, the only case we need, any algebraic group over~$k$ is smooth (theorem of Cartier, extended to the super world.)

\subsection{}\label{2.15} Suppose that $V$ is a linear representation of the algebraic super group~$G$, meaning that $\rho : G \to \GL(V)$ is given. A vector~$v$ in~$V$ is said to be \emph{$G$-invariant} if for any $R$, the image of~$v$ in $V_R$ is $G(R)$-invariant. If $G = \spec(A)$, then ``$G$ is a group scheme'' translates to ``$A$ is a (super) Hopf algebra.'' Let $\id$ be the $A$-point of~$G$ corresponding to the identity map $G = \spec(A) \to G$. Then, for any $R$, an $R$-point $g \in G(R) : \spec(R) \to G$ is deduced from $\id$ by functoriality: $g$ is the image of $\id$ by $g^* : A \to R$. Invariance of~$v$ can hence be checked just in the universal case $R = A$, $g = \id$.

The action of~$G$ is determined by the action of~$\id$, or equivalently by the comodule structure
\[ V \longrightarrow V_A \xrightarrow{\text{action of~$\id$}} V_A = V \otimes A. \]
The vector~$v$ is fixed by~$G$ if and only if it is fixed by~$\id$, that is, is in the kernel of the double arrow
\[\text{(extension of scalars, comodule structure)}: V \rightrightarrows V \otimes A. \]

\subsection{}\label{2.16} An open subscheme~$U$ of a scheme~$X$ (see \ref{2.1}) is \emph{schematically dense} in~$X$ if for any~$V$ open in~$X$ and for any $f$ in $\Gamma(V, \scrO)$, the vanishing of the restriction of~$f$ to~$U \cap V$ implies that $f = 0$. Example: if $X$ is smooth, for instance is an affine space, any Zariski dense open subscheme is schematically dense. 
\begin{lemma}
	\label{2.17} Suppose $X = \spec(R)$, with $R$ of finite type over~$k$, and that the $f_i$ in $R_0$ $(1 \leq i \leq n)$ define a closed subscheme whose complement is $U$. Then $U$ is schematically dense in~$X$ if and only if 
	\begin{equation}
		\label{2.17.1} R \to R^n : 1 \mapsto (f_1,\dotsc, f_n) 
	\end{equation}
	is injective. 
\end{lemma}
\begin{proof}
	We will show that \eqref{2.17.1} is not injective if and only if there exists $g \neq 0$ in~$R$ vanishing on~$U$.
	
	If $f \neq 0$ is in the kernel of \eqref{2.17.1}, the image of~$f$ in~$R[f_i^{-1}]$ vanishes: $f$ vanishes on the open $U_i$ where $f_i$ is invertible, hence on the union~$U$ of the $U_i$.
	
	Conversely, suppose that $f \neq 0$ vanishes on~$U$. This means that the image of~$f$ in each $R[f_i^{-1}]$ vanishes, that is, that for some~$\ell$, $ff_i^\ell = 0$. Let $I$ be the ideal generated by the $f_i$. For some $\ell$, $I^\ell f = 0$. Let $m \geq 0$ be the largest integer such that $I^m f \neq 0$, and pick $g \neq 0$ in~$I^m f$. By construction, $Ig = 0$: $g$ is in the kernel of~\eqref{2.17.1}. 
\end{proof}
\begin{cor}
	\label{2.18} If $f : X \to Y$ is flat and if $U$ is schematically dense in~$X$, then $f^{-1}(U)$ is schematically dense in~$Y$. 
\end{cor}
\begin{proof}
	The question is local: we may assume that $X = \spec(R)$ and that $Y = \spec(R')$, with~$R'$ flat over~$R$. By flatness, if \eqref{2.17.1} is injective, the same holds for~$R'$. 
\end{proof}
\begin{prop}
	\label{2.19} Suppose that $f : X \to Y$ is a faithfully flat morphism and that $V$, open in~$Y$, is schematically dense. Let $s$ be a function on~$V$: $s \in \Gamma(V,\scrO)$. If the pull back of~$s$ to~$f^{-1}(V)$ extends to~$X$, then $s$ extends to~$Y$. 
\end{prop}
\begin{proof}
	Let $t$ in $\Gamma(X, \scrO)$ be an extension to~$X$ of~$f^*(s)$ on~$f^{-1}(V)$. Let us consider
	\[ X \times_Y Y \rightrightarrows X \to Y. \]
	By schematic density in~$X \times_Y X$ of the inverse image~$V_1$ of $V$ \eqref{2.18}, one has $\mathrm{pr}_1^* t = \mathrm{pr}_2^* t$: both pull backs are extensions to~$X \times_Y X$ of the inverse image on~$V_1$ of~$s$. By faithfully flat descent, this implies that~$t$ is the pull back of some $\tilde s$ on~$Y$ and this $\tilde s$ is the required extension of~$s$ to~$Y$. 
\end{proof}
\begin{cor}
	\label{2.20} The same holds under the weaker assumption that $X$ covers $Y$, for the faithfully flat topology. 
\end{cor}
\begin{proof}
	The assumption is that there exists a commutative diagram
	\[ 
	\begin{tikzcd}
		X' \rar{u} \drar[swap]{f'} &X \dar{f}\\
		&Y 
	\end{tikzcd}
	\]
	with~$X'$ faithfully flat over~$Y$. If $t$ on~$X$ extends $f^*(s)$ on~$f^{-1}(V)$, then $u^*(f)$ on~$X'$ extends ${f'}^*(s)$ on~${f'}^{-1}(V)$, and one applies \ref{2.19} to~$f'$. 
\end{proof}

\subsection{}\label{2.21} If $X$ is smooth, and if the complement of the open subscheme~$U$ is of codimension $\geq 2$, one has
\[ \Gamma(X, \scrO) \xrightarrow{\sim} \Gamma(U, \scrO). \]
The proof is by reduction to the same statement for ordinary schemes, using that the ordinary scheme~$X_{\mathrm{rd}}$ is smooth and that locally $X$ can be retracted to~$X_{\mathrm{rd}}$, making $\scrO_X$ a free module over~$X_{\mathrm{rd}}$.

\section{Reduction of the FFT for $\osp$ to the FFT for~$\GL$}

From now on, the ground field is $\dbC$. The adjective ``super'' will be omitted.

\subsection{}\label{3.1} Let $V$ be a vector space of dimension~$m|2n$. As in~\ref{2.3}, suppose $V$ is in duality with~$W$. The \emph{space $\bar Q$ of quadratic forms} \eqref{2.7} on~$V$ is the affine space attached to~$\sym^2(W)$ \eqref{2.3}. Its $R$-points are the even functions~$q$ on $\bfV_R$ homogeneous of degree~$2$.

Choose a basis~$(e_i)$ of~$V$ and let $p(i)$ be the parity of~$e_i$. In coordinates, the $R$-points~$q$ of~$\bar Q$ can be written uniquely in the form
\[ q(x) = \sum a_{ij} x^i x^j \]
with~$a_{ij}$ in~$R$ of parity $p(i) + p(j)$, and the sum being over $\{(i, j) \mid \text{$i \leq j$ and $i \neq j$ if $e_i$ is odd} \}$.

Let $Q \subset \bar Q$ be the open and dense subscheme of non degenerate quadratic forms. The $R$-points of~$Q$, in coordinates, are the $\sum a_{ij} x^i x^j$ such that the associated bilinear form
\[ B(x,y) = \sum a_{ij}(x^iy^j + (-1)^{p(i)p(j)} x^j y^i) \]
is induced by a self-duality of $V_R$. It follows from~\ref{2.13} that $Q$ is a homogeneous space of~$GL(V)$. More precisely, if $q$ and $q'$ are two $S$-points of~$Q$, then, locally for the etale topology on~$S$, there exists an $S$-point $g$ of~$\GL(V)$ such that $g(q) = q'$, meaning that $q(g^{-1}v) = q'(v)$ after any subsequent change of basis.

From now on, we fix a non degenerate quadratic form~$q$ on~$V$, and write $B$ for the associated bilinear form~\eqref{2.7}. The map~$g \mapsto g(q)$ induces 
\begin{equation}
	\label{3.1.1} \GL(V)/\osv \xrightarrow{\sim} Q. 
\end{equation}
The $\GL(V)$-homogeneity of~$Q$ explained above means that \eqref{3.1.1} is in isomorphism, when both sides are interpreted, via the functor of points, as sheaves on the site of schemes of finite type over~$\dbC$, with the etale topology. This isomorphism implies that
\[ g \mapsto g(q) = q \circ g^{-1} : \GL(V) \to Q \]
is faithfully flat, with~$\osv$ as fiber above~$q$.

\subsection{}\label{3.2} Our aim is a description of the~$\osv$-invariant tensors on~$V$. Because the self-duality~$B$ is $\osv$-invariant, it suffices to describe the $\osv$-invariant elements\\
in~$\scrhom(V^{\otimes N}, \dbC)$. Because $\osv$ contains the parity automorphism of~$V$ ($1$ on~$V_0$ and $-1$ on~$V_1$), any such invariant element is even. Because $-1$ is in~$\osv$, there can be non zero invariant elements only when $N$ is even. Using~\ref{2.6}, the problem can be rephrased as describing the~$\osv$-invariant even functions on~$\bfV^N$ of degree one in each variable.

From~\eqref{3.1.1}, we get 
\begin{lemma}
	\label{3.3} Evaluation at~$q$ induces an isomorphism from $\GL(V)$-invariant functions on~$\bfV^N \times Q$ to $\osv$-invariant functions on~$\bfV^N$. 
\end{lemma}

The same holds when one considers only even functions of degree one in each variable~$v$ in~$\bfV$.

In other words: functions $f(v_1,\dotsc, v_N)$ on~$\bfV^N$, invariant under~$\osv$, extend uniquely to functions $F(v_1,\dotsc,v_N, q')$ on~$\bfV^N \times Q$, invariant under~$\GL(V)$. Invariance means that for any~$R$, and $R$-points $v_1,\dotsc,v_N$ in~$\bfV(R)$, $q'$ in~$Q(R)$ and $g$ in~$\GL(V)(R)$, one has $F(v_1,\dotsc,v_N, q') = F(gv_1,\dotsc,gv_N, g(q'))$. As $g(q') = q' \circ g^{-1}$, this can be restated as 
\begin{equation}
	\label{3.3.1} F(v_1,\dotsc,v_N, q' \circ g) = F(gv_1,\dotsc, gv_N, q'). 
\end{equation}
Special case, for $q' = q$: 
\begin{equation}
	\label{3.3.2} F(v_1,\ldots,v_N,q \circ g) = f(gv_1,\dotsc, gv_N). 
\end{equation}
\begin{thm}
	\label{3.4} Let $f$ be a $\osv$-invariant function on~$\bfV^N$, and let $F$ be the corresponding $\GL(V)$-invariant function on~$\bfV^N \times Q$, for which \eqref{3.3.1}, \eqref{3.3.2} hold. Then, $F$ extends to a function~$\bar F$ on~$\bfV^N \times \bar Q.$ 
\end{thm}
\begin{rems}
	\label{3.5} \leavevmode 
	\begin{enumerate}
		[(i)] 
		\item $Q$ is open non empty in the affine space~$\bar Q$, hence schematically dense. It follows that the function~$\bar F$ whose existence is promised by~\ref{3.4} is unique. 
		\item Define $E$ to be the affine space~$\bfhom(V,V)$ of~\ref{2.5}. As $\GL(V)$ is open non empty in~$E$, it is schematically dense in~$E$, and the identity~\eqref{3.3.1}, which is an identity on~$\bfV^N \times Q \times \GL(V)$ extends to an identity 
		\begin{equation}
			\label{3.5.1} \bar F(v_1,\dotsc, v_N, q' \circ T) = \bar F(Tv_1,\dotsc, Tv_N, q') 
		\end{equation}
		for~$v_i$ in~$\bfV$, $q'$ in~$\bar Q$ and $T$ in~$E$. Special case: $\bar F$ is $\GL(V)$-invariant. 
		\item Similar density arguments show that if $f$ is even (resp. homogeneous of degree one in each variable~$v_i$), so is $\bar F$. Under the same assumption, taking $T$ scalar, one gets for $r_0 \in R_0$
		\[ \bar F(v_1, \dotsc, v_N, r_0^2 q') = r_0^N \bar F(v_1,\dotsc, v_N, q'), \]
		which implies that $\bar F$ is homogeneous of degree~$N/2$ in~$q'$. 
	\end{enumerate}
\end{rems}

To prove theorem~\ref{3.4}, we require several lemmas. Let $u: E = \bfhom(V,V) \to \bar Q$ be the morphism $T \mapsto q \circ T$. One has $u^{-1}(Q) = \GL(V)$: 
\begin{equation}
	\label{3.5.2} 
	\begin{tikzcd}
		\GL(V) \rar[hook] \dar{u} &E \dar{u}\\
		Q \rar[hook] &\bar Q 
	\end{tikzcd}
	\quad. 
\end{equation}
\begin{lemma}
	\label{3.6} Let $\tilde F$ be the function on~$\bfV^N \times \GL(V)$ which is the inverse image by~$u$ of the function~$F$ on~$\bfV^N \times Q$. Then, $\tilde F$ extends to a function~$\bar{\tilde F}$ on~$\bfV_N \times E$. 
\end{lemma}

As the arguments of~\ref{3.5} show, this extension~$\bar{\tilde F}$ is unique. 
\begin{proof}
	By~\eqref{3.3.2}, for $T$ in~$\GL(V)$
	\[ \bar F(v_1,\dotsc,v_N,T) = F(v_1,\dotsc,v_N, q \circ T) = f(Tv_1,\dotsc,Tv_N). \]
	The right hand side continues to make sense for~$T$ in~$E$, and provides the required extension of~$\tilde F$. 
\end{proof}

\subsection{}\label{3.7} The condition on~$q'$ in~$\bar Q$: ``there exists in~$V$ a subspace $V_1 \subset V$ of codimension~$(1|0)$ on which $q'$ is non degenerate'' is an open condition. Let $\bar Q_1$ be the open subscheme of~$\bar Q$ where it holds. Similarly, let $\bar Q_2$ be the open subscheme where there exists $V_2 \subset V$ of codimension~$0|2$ on which $q'$ is non degenerate. 
\begin{lemma}
	\label{3.8} The complement in~$\bar Q$ of $\bar Q_1 \cup \bar Q_2$ is of codimension~$\geq 2$. 
\end{lemma}

This is a topological question, hence concerns only the reduced scheme~$\bar Q_{\mathrm{rd}}$, which is the product of the ordinary schemes parametrizing a quadratic form on $V^+ := V_0$ and an alternating form on $V^- := V_1$. These schemes are stratified by the rank, the stratum after the open stratum in an irreducible hypersurface, and we are removing the next one. 
\begin{lemma}
	\label{3.9} For the faithfully flat topology, $u^{-1}(\bar Q_1) \subset E$ covers $\bar Q_1$. 
\end{lemma}

This means there exists a faithfully flat morphism $S \to \bar Q_1$, which factors through the morphism $u^{-1}(\bar Q_1) \to \bar Q_1$. In fact, as we will see in~\ref{3.11}, the proof will show that $u^{-1}(\bar Q_1)$ is itself faithfully flat over~$\bar Q_1$. 
\begin{proof}
	It suffices to show that for any scheme~$S$, any subbundle (locally direct factor) $\scrV_1$ of the pull back $\scrV = \scrO_S \otimes V$ of~$V$ to~$S$, of codimension~$1|0$, and any quadratic form~$q'$ on~$\scrV$, non degenerate on~$\scrV'$, there exists a faithfully flat $S' \to S$ such that, after pull back to~$S'$, $q'$ is of the form~$q \circ T$ for an endomorphism~$T$ of~$\scrV$. Indeed, once this is proven, one can apply it locally to $S = \bar Q_1$.
	
	Because $q'$ is non degenerate on~$\scrV'$, $\scrV$ is the direct sum of~$\scrV'$ and of its orthogonal, relative to the bilinear form~$B'$ associated to~$q'$. Indeed, this orthogonal~${V'}^\perp$ is the kernel
	\[ 0 \longrightarrow {\scrV'}^\perp \longrightarrow \scrV \longrightarrow {\scrV'}^\vee \longrightarrow 0 \]
	of a morphism defined by~$B'$, for which $\scrV' \subset \scrV$ maps isomorphically to~${\scrV'}^\vee$.
	
	Let us fix $V_1 \subset V$ of codimension~$1|0$ on which $q$ is non degenerate, and let $\scrV_1$ be the corresponding subbundle of~$\scrV$. Let us look for a $T$ mapping $\scrV$ to $\scrV'$ and ${\scrV'}^\perp$ (orthogonal for~$q'$) to~$\scrV'$ (orthogonal for~$q$). This means finding an isomorphism from~$(\scrV', q')$ to~$(\scrV_1, q)$ and a morphism from~${\scrV'}^\perp$ to~$\scrV_1^\perp$ compatible with $q'$~and~$q$. The first exists locally for the etale topology \eqref{2.13}. The second is essentially the original problem in dimension~$1|0$.
	
	Indeed, locally on~$S$, $\scrV_1^\perp$ is $\scrO$, with a quadratic form~$ax^2$ with $a$ invertible while ${\scrV'}^\perp$ is locally isomorphic to~$\scrO$, with a quadratic form~$bx^2$. We need to find a morphism from~$\scrO$ to~$\scrO$, that is an even function $f$ such that $f^2 = b/a$. This equation defines a flat covering of~$S$, on which $f$ exists. 
\end{proof}
\begin{lemma}
	\label{3.10} For the flat topology, $u^{-1}(\bar Q_2) \subset E$ covers $\bar Q_2$. 
\end{lemma}

The proof follows the same step, reducing us this time to dimension~$0|2$. We have now to consider $\scrE$, with an odd basis $e_1, e_2$, and a non degenerate~$q$, written in coordinates as $ax_1x_2$ with $a$ invertible, and $\scrE'$, with an odd base $e_1', e_2'$ and $q'$ of the form~$bx_1'x_2'$. We need to find $T: \scrE' \to \scrE$ such that $q' = q \circ T$. Writing $T$ as a $2 \times 2$ matrix, this means solving an equation $\det T = b/a$. One can either argue that this equation has the solution~$ \left( 
\begin{smallmatrix}
	b/a & 0\\
	0 &1 
\end{smallmatrix}
\right)$, or that it defines $S'/S$ faithfully flat. 
\begin{rem}
	\label{3.11} The proof can be rephrased as expressing $u : u^{-1}(\bar Q_i) \to \bar Q_i$ as a composite
	\[ u^{-1}(\bar Q_i) \xrightarrow{v} Y \xrightarrow{w} \bar Q_i \]
	where a point of~$Y$ above $q'$ is a subspace of~$V$ of codimension $1|0$~or~$0|2$ on which $q'$ is non degenerate, and where $v(T)$ is the image~$T(V_i)$. Both maps are faithfully flat. 
\end{rem}
\begin{pf34}
	\label{3.12} By \ref{3.6},~\ref{3.9},~\ref{3.10} and~\ref{2.20}, the function~$F$ on~$\bfV^N \times Q$ extends to~$\bfV^N \times (\bar Q_1 \cup \bar Q_2)$. As the complement of $\bar Q_1 \cup \bar Q_2$ in~$\bar Q$ is of codimension~$\geq 2$ \eqref{3.8}, it extends to~$\bfV^N \times \bar Q$ \eqref{2.21}. \qed 
\end{pf34}

\subsection{FFT: reduction from $\osp$ to~$\GL$}\label{3.13}

Suppose $f : V^{\otimes N} \to \dbC$ is $\osp(V, q)$-invariant. We may suppose $N$ even. Interpret $f$ as a function on~$\bfV^N$, even and of degree one in each variable. By \ref{3.4}, \ref{3.5}, there exist a $\GL(V)$-invariant function $F(v_1,\dotsc,v_N, q')$ on~$\bfV^N \times \bar Q$, even of degree one in the $v_i$, and of degree~$N/2$ in~$q'$, such that
\[ f(v_1,\dotsc,v_N) = F(v_1,\dotsc,v_N,q). \]
The usual polarization argument shows that there exists a $\GL(V)$-invariant function\\
$F_1(v_1,\dotsc,v_N,q_1',\dotsc,q_{N/2}')$ on $\bfV^N \times \bar Q^{N/2}$ homogeneous in degree one in each variable, such that
\[ F(v_1,\dotsc,v_N, q) = F_1(v_1,\dotsc,v_N, q,\dotsc, q). \]
As in~\ref{2.3}, let $W$ be in duality with~$V$. Interpret $F_1$ as a morphism
\[ V^{\otimes N} \otimes \sym^2(W)^{\otimes N/2} \to \dbC, \]
invariant by~$\GL(V)$. As the decomposition
\[ W \otimes W = \sym^2(W) \oplus \wedge^2 W \]
is $\GL(V)$-invariant, $F_1$ comes from a $\GL(V)$-invariant morphism
\[ V^{\otimes N} \otimes W^{\otimes N} \to \dbC. \]
This is the announced reduction.

The FFT for~$\GL(V)$ says that a generating set for the $\GL(V)$-invariant morphisms is obtained by pairing in all possible ways each factor~$V$ with a factor~$W$, and taking the product:
\[ v_1,\dotsc,v_N, w_1,\dotsc,w_N \to \prod \langle v_i, w_{\sigma(i)}\rangle. \]

Passing back to the $\osp$-invariant, we find a generating set consisting of
\[ V^{\otimes N} \to \dbC : v_1,\dotsc, v_N \mapsto B(v_1,v_2)B(v_3,v_4) \dotsm B(v_{N-1}, v_N) \]
and its transforms by the symmetric group.

\section{Periplectic analog}

\subsection{}\label{4.1} Let $V$ be a vector space of dimension~$n|n$. An \emph{odd quadratic form}~$q$ on~$V$ is an odd function of degree~$2$ on~$\bfV$. Equivalently, if $V$~and~$W$ are in duality, it is an odd element of~$\sym^2(W)$. The associated bilinear form~$B$ is the odd function $q(x + y) - q(x) - q(y)$ on~$\bfV \times \bfV$. If $\Pi$ is $\dbC$, viewed as an odd vector space with basis~$\bfe := 1$, so that $\mathbf{\Pi} = \dbG^-_a$, $B$ is induced \eqref{2.6} by a (super) symmetric morphism again noted $B$: 
\begin{equation}
	\label{4.1.1} V \otimes V \to \Pi. 
\end{equation}

The form~$q$ is \emph{non degenerate} if \eqref{4.1.1} is a perfect pairing, that is induces an isomorphism of~$V$ with $\Pi V^\vee := \Pi \otimes V^\vee$. Suppose $q$ is non degenerate. The \emph{periplectic group}~$\pi\osv$ is the subgroup of~$\GL(V)$ respecting~$q$ or, equivalently, $B$. The description of $R$-points is the same as in~\ref{2.11}.

As in \ref{2.9}~and~\ref{2.12}, those definitions can be repeated over a scheme~$S$. The analog of~\ref{2.13} is the 
\begin{prop}
	\label{4.2} Let $q$ be an odd non degenerate quadratic form on a vector bundle~$\scrV$ on~$S$. Let $B$ be the associated bilinear form
	\[ B: \scrV \otimes \scrV \to \Pi \otimes \scrO. \]
	Then, locally on~$S$, for some~$n$, $\scrV$ admits a basis $e_1,\dotsc, e_n, f_1,\dotsc,f_n$, with even $e_i$ and odd $f_i$, such that on this basis $B$ vanishes, except for
	\[ B(e_i, f_i) = B(f_i, e_i) = \bfe. \]
\end{prop}
\begin{proof}
	As in~\ref{2.13}, we may assume $\scrV$ of some dimension~$m|n$. Because $\scrV$ is isomorphic to~$\scrV^\vee \otimes \Pi$, one has $m = n$. The proof is by induction on~$n$, the case~$n = 0$ being trivial. It will be convenient to view $B$ as an odd bilinear map: $\text{(previous $B$)} = \text{($\bfe$. this new $B$)}$.
	
	If $n = 1$, let $(e, f)$ be a basis of~$\scrV$ with $e$ even, $f$ odd. The symmetry of~$B$ implies that $B(f,f) = 0$. The non degeneracy then implies that $B(e,f)$ is even and invertible. Dividing $f$ by~$B(e, f)$, we may assume that $B(e, f) = 1$. As $B$ is odd, so is $\alpha := B(e,e)$. The basis~$(e - \frac{\alpha}{2}f, f)$ is of the announced type. If $n \geq 1$, one finds as in~\ref{2.13} even and odd local sections $e$~and~$f$ such that $B(e, f)$ is invertible. The bundle~$\scrV$ is then the orthogonal direct sum of~$\scrV'$, spanned by $e$~and~$f$, and of its orthogonal complement~$\scrV''$. One concludes by applying the induction hypothesis to~$\scrV''$. 
\end{proof}

\subsection{}\label{4.3} Let $\bar Q$ be the space of odd quadratic forms on~$V$. It is the affine space attached to~$\sym^2 W \otimes \Pi$. Let $Q$ be the open and dense subscheme of non degenerate quadratic forms. As in~\ref{3.1}, it follows from~\ref{4.2} that it is a homogeneous space of~$\GL(V)$. Let us fix a non degenerate odd quadratic form~$q$ on~$V$, with the associated odd bilinear form~$B$. As in~\ref{3.1} the map
\[ g \mapsto g(q) := q \circ g^{-1} : \GL(V) \to Q \]
is flat, with fiber~$\pi\osv$ at~$q$, and induces an isomorphism 
\begin{equation}
	\label{4.3.1} \GL(V)/\pi\osv \xrightarrow{\sim} Q. 
\end{equation}

\subsection{}\label{4.4} Our aim is a description of the $\pi\osv$-invariant tensors on~$V$. Because $B$ is a perfect pairing, it suffices to describe the $\pi\osv$-invariant elements in~$\scrhom(V^{\otimes N}, \dbC)$. Because $-1$ is in~$\pi\osv$, non zero invariants can exist only for~$N$ even. Using \ref{2.6}, the problem can be rephrased as describing the $\osv$-invariant functions on~$\bfV^N$, of degree one in each variable.

From~\eqref{4.3.1}, we get that evaluation at~$q$ induces an isomorphism from $\GL(V)$-invariant function on~$\bfV^N \times Q$ and $\pi\osv$-invariant functions on~$\bfV^N$, and the same holds when considering only functions of degree one each variable~$v$ in~$\bfV$. The key result (cf.~\ref{3.4}) is: 
\begin{thm}
	\label{4.5} Let $f$ be a $\pi\osv$-invariant function on~$\bfV^N$, and let $F$ be the corresponding $\GL(V)$-invariant function on~$\bfV^N \times Q$. Then, $F$ extends to a function~$\bar F$ on~$\bfV^N \times \bar Q$. 
\end{thm}

The remarks~\ref{3.5} continue to hold, mutatis mutandis. Let $u : E = \bfhom(V,V) \to \bar Q$ be morphism $T \mapsto q \circ T$. One has $u^{-1}(Q) = \GL(V)$: 
\begin{equation}
	\label{4.5.1} 
	\begin{tikzcd}
		\GL(V) \rar[hook] \dar{u} &E \dar{u}\\
		Q \rar[hook] &\bar Q 
	\end{tikzcd}
	\quad. 
\end{equation}
Let $\tilde F$ be the function on~$\bfV^N \times \GL(V)$ which is the inverse image by~$u$ of~$F$. The same proof as in~\ref{3.6} shows that $\tilde F$ extends to a function~$\bar{\tilde F}$ on~$\bfV^N \times E$.

Imitating the proof of~\ref{3.4}, one deduces~\ref{4.5} from the lemmas \ref{4.7}~and~\ref{4.8} below.

\subsection{}\label{4.6} The condition on~$q'$ in~$\bar Q$: ``there exists in~$V$ a subspace~$V'$ of codimension~$1|1$ on which $q$ is non degenerate'' is an open condition. Let $\bar Q_1$ be the open subscheme of~$\bar Q$ where it holds. 
\begin{lemma}
	\label{4.7} The complement of~$\bar Q_1$ in~$\bar Q$ is of codimension~$\geq 2$. 
\end{lemma}
\begin{proof}
	This is a topological question, hence concerns only the reduced scheme~$\bar Q_{\mathrm{rd}}$, which is the space of bilinear pairings between $V_0$~and~$V_1$. This space is stratified by rank. The stratum after the open strata is an irreducible hypersurface, and we are considering the next one. 
\end{proof}
\begin{lemma}
	\label{4.8} For the faithfully flat topology, $u^{-1}(\bar Q_1) \subset E$ covers $\bar Q_1$. 
\end{lemma}

Here, it is not true that $u^{-1}(\bar Q_1) \subset E$ is flat over~$\bar Q_1$. But we will show that there exist local sections, so that the existence of~$\bar{\tilde F}$ trivially implies that $F$ on~$\bfV^n \times Q$ extends to~$\bfV^N \times Q_1$. 
\begin{proof}
	Proceeding as in the proof of~\ref{3.9}, one has to show that if on~$S$ we have vector bundles $\scrV$~and~$\scrV'$ of dimension~$1|1$, a non degenerate odd quadratic form~$q$ on~$\scrV$, and an odd quadratic form~$q'$ on~$\scrV'$, then locally on~$S$, there exists $T: \scrV' \to \scrV$ such that $q' = q \circ T$.
	
	In coordinates, in a suitable local basis, $q$ can be written as
	\[ q(x, y) = xy \qquad \text{($x$ even, $y$ odd)} \]
	while $q'$ can be written
	\[ q'(x, y) = \alpha x + axy \qquad \text{($x$ even, $y$ odd, $\alpha$ even, $a$ even).} \]
	One takes $T: (x, y) \mapsto (x, \alpha x + a y)$. 
\end{proof}

\subsection{FFT: reduction from $\pi \mathrm{O}$ to $\GL$}\label{4.9} One proceeds as in~\ref{3.13} to deduce from~\ref{4.5} and from the first fundamental theorem for~$\GL(V)$ that the $\pi\osv$-invariant multilinear forms on~$V^{\otimes N}$, ($N$ even), with the values in $\dbC$~or~$\Pi$, are the form
\[ v_1,\dotsc, v_N \mapsto B(v_1,v_2) \dotsm B(v_{N-1},v_N) \]
and its transforms under the symmetric group.

\section{Super Pfaffian}

\subsection{}\label{5.1} Let $q$ be a non degenerate quadratic form on a vector space~$V$ of dimension~$m|2n$, and let $B : V \otimes V \to \dbC$ be the associated (super) symmetric bilinear form. Let $\vol$ be one of the two opposite isomorphisms of~$\bigwedge^m V_0$ with~$\dbC$ such that for~$v_1,\dotsc,v_m$ in~$V_0$ one has 
\begin{equation}
	\label{5.1.1} \vol(v_1 \wedge \dotsb \wedge v_m)^2 = \det B(v_i, v_j) 
\end{equation}
We will view $\vol$ as an alternating multilinear form on~$V_0$. As such, it defines a function, still noted $\vol$, on $\bfV_0^m$ \eqref{2.6}. Similarly, one views $B(v_i, v_j)$ as a function on~$\bfV^m$.

The scheme~$\bfV_0$ is $\bfV_{\mathrm{rd}}$, and is the reduced subscheme of~$\bfV$. It is defined in~$\bfV$ by a nilpotent ideal. Let $U$ be the open subscheme of~$\bfV^m$ on which $\det B(v_i, v_j)$ is invertible. 
\begin{prop}
	\label{5.2} On $U$, the function $D := \det B(v_i, v_j)$ has a unique square root~$\Delta$ which, on~$U_{\mathrm{rd}} \subset \bfV_0$, agrees with $\vol(v_1,\dotsc,v_m)$. 
\end{prop}
\begin{proof}
	Let $\scrR$ be the double covering of~$\bfV^m$ defined by the equation $T^2 = D$. An $R$-point of~$\scrR$ above an $R$-point $(v_1,\dotsc,v_m)$ of~$\bfV^m$ is a square root in~$R_0$ of $\det B(v_i, v_j)$. Over $U$, this double covering is etale. A section of~$\scrR$ above~$U_{\mathrm{rd}}$ hence uniquely extends to~$U$. This is applied to the restriction to~$U_{\mathrm{rd}}$ of $\vol(v_1,\dotsc,v_m)$ on~$\bfV_0^m$. 
\end{proof}

\subsection{}\label{5.3} Suppose we have $U$ schematically dense in a scheme~$X$, and a function~$\delta$ on~$U$ such that 
\begin{enumerate}
	[a)] 
	\item the restriction of~$\delta$ to the reduced scheme~$U_{\mathrm{red}}$ extends to~$X_{\mathrm{red}}$; 
	\item the extension vanishes outside of~$U_{\mathrm{red}}$, that is, its invertibility locus is contained in~$U_{\mathrm{red}}$. 
\end{enumerate}
Then, any large enough power~$\delta^N$ of~$\delta$ extends to~$X$.

In the case of~$\Delta$, we want to prove a more precise statement. 
\begin{thm}
	\label{5.4} The power~$\Delta^{2n+1}$ of~$\Delta$ extends to~$\bfV^m$. 
\end{thm}
\begin{proof}
	[Proof when $m = 1$] Let us work in coordinates, using a basis in which the function~$B(v,v)$ on~$\bfV$ has a standard form: if the coordinates of~$v$ are $x$, even, and $y_i$ ($1 \leq i \leq 2n$), odd, we assume that 
	\begin{equation}
		\label{5.4.1} B(v,v) = x^2 + \sum_1^n y_i y_{i+n}. 
	\end{equation}
	For a suitable choice of $\vol$,~$\Delta$, defined for $x$ invertible, is then the square root of~$B(v,v)$, which reduces to~$x$ when the $y_i$ vanish. It is given by the binomial expansion\\
	of~$(x^2 + \sum y_iy_{i+n})^{1/2}$: 
	\begin{equation}
		\label{5.4.2} \Delta = \sum \binom{1/2}{k} x^{1-2k} \left( \sum_1^m y_i y_{i+m} \right)^k. 
	\end{equation}
	Because $\left( \sum_1^n y_i y_{i+m} \right)^k = 0$ for~$k > n$, this expansion terminates.
	
	Similarly, 
	\begin{equation}
		\label{5.4.3} \Delta^{2n+1} = \sum_{k = 0}^n \binom{n + 1/2}{k} x^{2n + 1 - 2k}. \left( \sum_1^n y_i y_{i+m} \right)^k, 
	\end{equation}
	which is a polynomial in~$x$ and the $y_i$, meaning that $\Delta^{2n+1}$ extends to~$\bfV$. 
\end{proof}
\begin{pf54}
	\label{5.5} Define
	\[ D_1 := \det(B(v_i, v_j)_{1 \leq i,\ j \leq m-1}). \]
	and let $U_1$ be the open subset of~$\bfV^m$ where $D_1$ is invertible. 
\end{pf54}
\begin{lemma}
	\label{5.6} The complement of~$U \cup U_1$ in~$\bfV^m$ is of codimension~$\geq 2$. 
\end{lemma}

Indeed, the complements of $U$ and $U_1$ are distinct irreducible hypersurfaces. 
\begin{lemma}
	\label{5.7} $\Delta^{2n+1}$ extends from~$U_1 \cap U$ to~$U_1$. 
\end{lemma}

This lemma implies that $\Delta^{2n+1}$ extends from $U$ to~$U \cup U_1$, while \ref{2.21}~and~\ref{5.6} then imply that it extends to~$\bfV^m$, proving \ref{5.4}. By \ref{2.19}, it suffices to prove \ref{5.7} locally for the etale topology on~$U_1$. 
\begin{proof}
	[Proof of~\ref{5.7}] Let $\scrV$ be the vector bundle~$V \otimes \scrO$ on~$\bfV^m$. On~$U_1$, the vector bundle~$\scrV$ is the orthogonal direct sum of a subbundle~$\scrV'$ with basis~$v_1,\dotsc,v_{m-1}$, and of its orthogonal complement~$\scrV''$, of dimension~$1|2n$. Indeed, this orthogonal complement of~$v_1,\dotsc,v_{m-1}$ is the kernel of $b : \scrV \to \scrO^{m-1} : v \mapsto (B(v, v_i)_{1 \leq i \leq m-1})$, and the composite
	\[ \scrO^{m-1} \xrightarrow{(v_1,\dotsc,v_{m-1})} \scrV \xrightarrow{b} \scrO^{m-1} \]
	is an isomorphism. 
	
	Decompose $v_m$ as $v_m' + v_m''$, with~$v_m'$ a linear combination of the $v_i$ $(1 \leq i \leq m-1)$, and~$v_m''$ in~$\scrV''$. One has
	\[ D = D_1 . B(v_m'',v_m''). \]
	
	Locally for the etale topology, that is after a pull back by~$u : U_1' \to U_1$ etale and surjective, 
	\begin{enumerate}
		[(a)] 
		\item $D_1$ has a square root $D_1^{1/2}$ (indeed, $D_1$ is invertible) 
		\item $\scrV''$ has a basis in which $B$ has the standard form~\eqref{5.4.1} (cf.~\ref{2.13}.) 
	\end{enumerate}
	
	Let $U_1''$ be any connected component of~$U_1'$, and let $(U \cap U_1)''$ be the trace on~$U_1''$ of the inverse image of $U \cap U_1$ in~$U_1'$. As $U_1'$ is a smooth, $U_1''$ is irreducible and $(U \cap U_1)''$, dense in~$U_1''$, is connected. We will show that (the pull back of) $\Delta^{2n + 1}$ extends from~$(U \cap U_1)''$ to~$U_1''$.
	
	On~$(U \cap U_1)''$, $\Delta / D_1^{1/2}$ is a square root~$B(v'',v'')^{1/2}$ of~$B(v'',v'')$. As $\Delta = D_1^{1/2}. B(v'',v'')^{1/2}$, it suffices to show that $[B(v'',v'')^{1/2}]^{2n+1}$ extends to~$U_1''$.
	
	Let $x$ and the $y_i$ be the coordinates of~$v''$. On~$U_1'$, $x$ is invertible. On~$(U \cap U_1)''$, which is connected, $B(v'',v'')$ has only two square roots, differing by a sign. It follows that $B(v'',v'')^{1/2}$ is given by~\eqref{5.4.2}, or its opposite. The desired extension of~$[B(v'',v'')^{1/2}]^{2n+1}$ is given by~\eqref{5.4.3}, or its opposite. 
\end{proof}

\subsection{}\label{5.8} Sergeev calls the square root~$\Delta^{2n + 1}$ of $(\det B(v_i, v_j))^{2n+1}$ the \emph{super Pfaffian.} The connected component~$\mathrm{SOSp}(V, q)$ of~$\osp(V, q)$ is the kernel of the Berezinian, with values in~$\mu_2 = \{\pm 1\}$. The super Pfaffian is invariant by~$\mathrm{SOSp}(V, q)$, because $\det B(v_i, v_j)$ is, and that $\mathrm{SOSp}$ is connected. The elements of the other connected component of~$\osp$ transform the super Pfaffian into its opposite. 

\subsection{}\label{5.9} In~\cite{lz15} theorem~\ref{5.4} is proved by giving a second definition of the super Pfaffian and showing it is equivalent to the one above. It is also proved in op.~cit. that the invariants of~$\mathrm{SOSp}(V, q)$ on~$S^*(V \otimes \dbC^N)$ are generated by the invariants of~$\osp(V, q)$, together with the super Pfaffian.

For this, the $\mathrm{gl}_N$-module structure of the space of invariants is relevant, and elucidated in~\cite{lz15}.

\newpage 
\begin{flushleft}
	P.~Deligne\qquad \url{deligne@math.ias.edu}
	
	School of Mathematics, IAS, Princeton, NJ 08540, USA
	
	\medskip
	
	G.~I.~Lehrer\qquad \url{gustav.lehrer@sydney.edu.au}
	
	School of Mathematics and Statistics, University of Sydney, NSW 2006, Australia
	
	\medskip
	
	R.~B.~Zhang\qquad \url{ruibin.zhang@sydney.edu.au}
	
	School of Mathematics and Statistics, University of Sydney, NSW 2006, Australia 
\end{flushleft}

\end{document}